\documentclass[11pt]{amsart}

\usepackage{latexsym}
\usepackage{amssymb}
\usepackage{amsmath}
\usepackage{fancybox,color}
\usepackage{graphicx}
\def\1{\raisebox{2pt}{\rm{$\chi$}}}

\newtheorem{theorem}{Theorem}
\newtheorem{corollary}[theorem]{Corollary}
\newtheorem{lemma}[theorem]{Lemma}

\newtheorem{definition}[theorem]{Definition}
\newtheorem{remark}[theorem]{Remark}
\newtheorem{example}[theorem]{Example}
\newcommand{\R}{{\mathbb R}}

\newcommand{\N}{{\mathbb N}}

 \newcommand{\eps}{{\varepsilon}}
 \def\1{\raisebox{2pt}{\rm{$\chi$}}}

\usepackage{enumerate}

\def\vint_#1{\mathchoice%
          {\mathop{\kern 0.2em\vrule width 0.6em height 0.69678ex depth -0.58065ex
                  \kern -0.8em \intop}\nolimits_{\kern -0.4em#1}}%
          {\mathop{\kern 0.1em\vrule width 0.5em height 0.69678ex depth -0.60387ex
                  \kern -0.6em \intop}\nolimits_{#1}}%
          {\mathop{\kern 0.1em\vrule width 0.5em height 0.69678ex
              depth -0.60387ex
                  \kern -0.6em \intop}\nolimits_{#1}}%
          {\mathop{\kern 0.1em\vrule width 0.5em height 0.69678ex depth -0.60387ex
                  \kern -0.6em \intop}\nolimits_{#1}}}
\def\vintslides_#1{\mathchoice%
          {\mathop{\kern 0.1em\vrule width 0.5em height 0.697ex depth -0.581ex
                  \kern -0.6em \intop}\nolimits_{\kern -0.4em#1}}%
          {\mathop{\kern 0.1em\vrule width 0.3em height 0.697ex depth -0.604ex
                  \kern -0.4em \intop}\nolimits_{#1}}%
          {\mathop{\kern 0.1em\vrule width 0.3em height 0.697ex depth -0.604ex
                  \kern -0.4em \intop}\nolimits_{#1}}%
          {\mathop{\kern 0.1em\vrule width 0.3em height 0.697ex depth -0.604ex
                  \kern -0.4em \intop}\nolimits_{#1}}}

\newcommand{\aveint}[2]{\mathchoice%
          {\mathop{\kern 0.2em\vrule width 0.6em height 0.69678ex depth -0.58065ex
                  \kern -0.8em \intop}\nolimits_{\kern -0.45em#1}^{#2}}%
          {\mathop{\kern 0.1em\vrule width 0.5em height 0.69678ex depth -0.60387ex
                  \kern -0.6em \intop}\nolimits_{#1}^{#2}}%
          {\mathop{\kern 0.1em\vrule width 0.5em height 0.69678ex depth -0.60387ex
                  \kern -0.6em \intop}\nolimits_{#1}^{#2}}%
          {\mathop{\kern 0.1em\vrule width 0.5em height 0.69678ex depth -0.60387ex
                  \kern -0.6em \intop}\nolimits_{#1}^{#2}}}

\newcommand{\ud}{\, d}

\newcommand{\abs}[1]{\left| #1 \right|}

\newcommand{\ol}{\overline}
\newcommand{\Om}{\Omega}
\newcommand{\I}{\textrm{I}}
\newcommand{\II}{\textrm{II}}
\newcommand{\dist}{\operatorname{dist}}

\begin{document}

\title[Games for eigenvalues of the Hessian]
{\bf Games for eigenvalues of the Hessian and concave/convex envelopes}

\author[P. Blanc and J. D. Rossi]{Pablo Blanc and Julio D. Rossi}

\address{Departamento  de Matem\'atica, FCEyN, Universidad de Buenos Aires \hfill\break\indent
\qquad and IMAS - CONICET,
\hfill\break\indent Pabell\'on I, Ciudad Universitaria (1428),
Buenos Aires, Argentina.}

\email{ {\tt pblanc@dm.uba.ar, }\qquad {\tt jrossi@dm.uba.ar}\hfill\break\indent {\it Web page:}
{\tt http://mate.dm.uba.ar/$\sim$jrossi}}



\date{}

\begin{abstract} 
We study the PDE 
$\lambda_j(D^2 u)  = 0$,  in $\Omega$, with
$u=g$,  on $\partial \Omega$.
Here $\lambda_1(D^2 u) \leq ... \leq \lambda_N (D^2 u)$
are the ordered eigenvalues of the Hessian $D^2 u$. 
First, we show a geometric interpretation of the viscosity solutions to the problem in terms of convex/concave
envelopes over affine spaces of dimension $j$. In one of our main results,
we give necessary and sufficient 
conditions on the domain so that the problem has a continuous solution for every continuous datum $g$.
Next, we introduce a two-player zero-sum game whose
values approximate solutions to this PDE problem.
In addition, we show an asymptotic mean value characterization for the solution the the PDE.
\end{abstract}

\maketitle

\section{Introduction}

In this paper, we study the boundary value problem
\begin{equation}
\label{1.1}
\tag{$\lambda_j,g$}
\left\{
\begin{array}{ll}
\lambda_j(D^2 u)  = 0 , \qquad & \mbox{ in } \Omega, \\[5pt]
u=g , \qquad & \mbox{ on } \partial \Omega.
\end{array}
\right.
\end{equation}
Here $\Omega$ is a domain in $\mathbb{R}^N$ and 
for the Hessian matrix of a function $u:\Omega \mapsto \mathbb{R}$, $D^2 u$, we denote by
$$\lambda_1(D^2 u) \leq ... \leq \lambda_N (D^2 u)$$
the ordered eigenvalues. Thus our equation says that the $j-$st smaller eigenvalue of the Hessian
is equal to zero inside $\Omega$.  

The uniqueness and a comparison principle for the equation were proved in \cite{HL1}.
For the existence, in \cite{HL1} it is assumed that the domain is smooth 
(at least $C^2$) and such that $\kappa_1 \leq \kappa_2 \leq ... \leq  \kappa_{N-1}$, the main curvatures of $\partial \Omega$, verify
\begin{equation} \label{cond.Geo}
\tag{H}
 \kappa_j (x) > 0 \mbox{ and } \kappa_{N-j+1} (x) >0, \qquad \forall x \in \partial \Omega.
\end{equation}

Our main goal here is to improve the previous result and 
give sufficient and necessary conditions on the domain 
(without assuming smoothness of the boundary)
so that the problem has a continuous solution for every continuous data $g$.
Our geometric condition on the domain reads as follows:
Given $y\in\partial\Omega$ we assume that for every
 $r>0$ there exists $\delta>0$ such that for every $x\in B_\delta(y)\cap \Omega$ and $S\subset\R^N$ a subspace of dimension $j$, there exists $v\in S$ of norm 1 such that
\begin{equation}
\label{condG1}
\tag{$G_j$} 
\{x+tv\}_{t\in \R}\cap B_r(y)\cap \partial\Omega\neq\emptyset.
\end{equation}
We say that $\Omega$ satisfies condition (G) if it satisfies both $(G_j)$ and $(G_{N-j+1})$.

\begin{theorem} \label{teo.main}
The equation \eqref{1.1} has a continuous solution for every continuous data $g$ if and only 
if $\Omega$ satisfies condition {\rm (G)}.
\end{theorem}

As part of the proof of this theorem we use the following geometric interpretation of solutions to \eqref{1.1}. 
Let $H_j$ be the set of functions $v$ such that 
$$v \leq g \qquad \mbox{on } \partial \Omega,$$
and have the following property: for every $S$ affine 
of dimension $j$ and every $j-$dimensional domain 
$D \subset S \cap \Omega$ it holds that 
$$
v \leq z \qquad \mbox{ in } D
$$
where $z$ is the concave envelope of $v|_{\partial D}$ in $D$. 
Then we have the following result:

\begin{theorem} \label{teo.convex.envelope.k.intro} The function
$$
u(x) = \sup_{v \in H_j} v(x).
$$
is the largest viscosity solution to $\lambda_j(D^2 u)  = 0 $,  in $ \Omega$, 
with $u\leq g$ on $\partial \Omega$.
\end{theorem}

Notice that, for $j=N$, we have that the equation
for the concave envelope of $u|_{\partial \Omega}$ in $\Omega$
is just $\lambda_N =0$; while the equation for the convex envelope 
is $\lambda_1 =0$. See \cite{OS}. Notice that our condition (G) in these 
two extreme cases
is just saying that the domain is strictly convex. Hence, Theorem \ref{teo.main}
implies that for a strictly convex domain the concave or the convex envelope
of a continuous datum $g$ on its boundary is attached to $g$ continuously.
Note that the concave/convex envelope of $g$ inside $\Omega$ is well defined for every domain
(just take the infimum/supremum of concave/convex functions that are above/below
$g$ on $\partial \Omega$). The main point of Theorem \ref{teo.main} is the continuity
up to the boundary of the concave/convex envelope of $g$ if and only if (G) holds. 
Remark that Theorem \ref{teo.convex.envelope.k.intro} 
says that the equation $\lambda_j(D^2 u)  = 0$ for $1<j<N$ is also related to
concave/convex envelopes of $g$, but in this case we consider concave/convex functions
restricted to affine subspaces. Also in this case Theorem \ref{teo.main} gives a necessary
and sufficient condition on the domain in order to have existence of a solution that is continuous
up to the boundary.

Remark that we have that $u$ is a continuous solution to \eqref{1.1} if and only if $-u$ is a solution to
$(\lambda_{N-j+1},-g)$. This fact explains why we have to include both $(G_j)$ and $(G_{N-j+1})$
in condition (G).

Our original motivation to study the problem \eqref{1.1} comes from game theory.
Let us describe the game that we propose to approximate solutions to the equation. 
It is a two-player zero-sum game. Fix a domain $\Omega \subset \mathbb{R}^N$, $\eps>0$ and a final payoff function 
$g :\mathbb{R}^N \setminus \Omega \mapsto \mathbb{R}$. 
The rules of the game are the following:
the game starts with a token at an initial position $x_0 \in \Omega$, 
one player (the one who wants to minimize the expected payoff)
chooses a subspace $S$ of dimension $j$ and then the second player
(who wants to maximize the expected payoff)
chooses one unitary vector, $v$, in the subspace $S$. Then the position of the token
is moved to $x\pm \epsilon v$ with equal probabilities. 
The game continues until the position of the token leaves the domain and at this 
point $x_\tau$ the first player gets $-g(x_\tau)$ and the second player $g(x_\tau)$.
When the two players fix their strategies $S_I$ (the first player chooses a $j-$dimensional subspace
$S$ at every step of the game) and $S_{II}$ (the second player chooses a unitary vector $v\in S$
at every step of the game) we can compute the expected outcome as
$$
\mathbb{E}_{S_I, S_{II}}^{x_0} [g (x_\tau)].
$$
Then the values of the game for any $x_0 \in \Omega$ 
for the two players are defined as
\[
u^\eps_\I(x_0)=\inf_{S_\I}\sup_{S_{\II}}\,
\mathbb{E}_{S_{\I},S_\II}^{x_0}\left[g(x_\tau)\right], \qquad
u^\eps_\II(x_0)=\sup_{S_{\II}}\inf_{S_\I}\,
\mathbb{E}_{S_{\I},S_\II}^{x_0}\left[g(x_\tau)\right].
\]
When the two values coincide we say that the game has a value.

Next, we state that this game has a value and the value verifies an equation
(called the Dynamic Programming Principle (DPP) in the literature).

\begin{theorem} \label{teo.dpp}
The game has value 
$$
u^\epsilon = u_{I}^\epsilon = u_{II}^\epsilon
$$
that verifies 
\begin{equation*}\label{eq.DPP}
\tag{DPP}
\left\{
\begin{array}{ll}
\displaystyle u^\epsilon (x) = \inf_{
{dim}(S)=j
} \sup_{
v\in S, |v|=1}
\left\{ \frac{1}{2} u^\epsilon (x + \epsilon v) + \frac{1}{2} u^\epsilon (x - \epsilon v)
\right\}  & x \in \Omega, \\[5pt]
u^\epsilon (x) = g(x)  & x \not\in \Omega.
\end{array}
\right.
\end{equation*}
\end{theorem}

Our next goal is to look for the limit as $\eps \to 0$. 
To this end we need
another geometric assumption on $\partial \Omega$. 
Given $y\in\partial\Omega$ we assume that there exists
 $r>0$ such that for every $\delta>0$ there exists $T\subset\R^N$ a subspace of dimension $j$, $v\in\R^N$ of norm 1, $\lambda>0$ and $\theta>0$ such that
\begin{equation} \label{cond.geoj}
\tag{$F_j$}
\{x\in\Omega\cap B_r(y)\cap T_\lambda: \langle v,x-y\rangle<\theta\}\subset B_\delta(y)
\end{equation}
where 
\[
T_\lambda=\{x\in\R^N: d(x,T)<\lambda\}.
\]
For our game with a given $j$ we will assume that $\Omega$ satisfies 
both ($F_j$) and ($F_{N-j+1}$), in this case we will say that $\Omega$ satisfy condition ($F$).

For example, if we consider the equation $\lambda_2=0$ in $\R^3$, we will require that the domain satisfy $(F_2)$ as illustrated in Figure~\ref{fig:condition}.

\begin{figure}
\includegraphics[scale=0.0625]{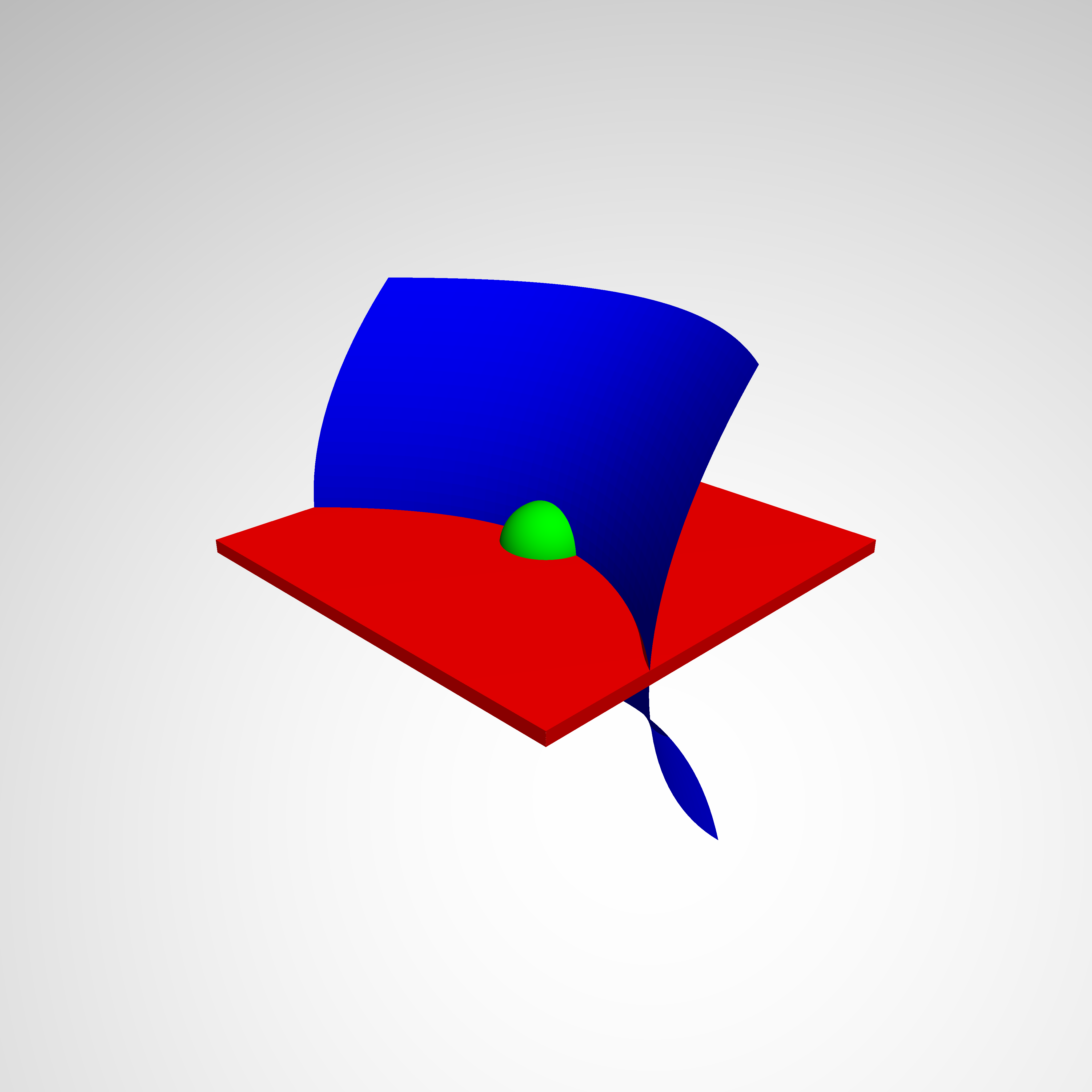}
\label{fig:condition}
\caption{Condition $(F_2)$ in $\R^3$. We have $\partial\Omega$ in blue, $B_\delta(y)$ in green and $T_\lambda$ in red.}
\end{figure}

\begin{theorem} \label{teo.converge}
Assume that $\Omega$ satisfies 
{\rm ($F$)} and let $u^\eps$ be the values of the game. Then,
\begin{equation*}
u^\eps \to u, \qquad \mbox{ as } \epsilon \to 0,
\end{equation*}
uniformly in $\overline{\Omega}$. Moreover, the limit $u$ is characterized as the unique viscosity
solution to
\begin{equation*}
\left\{
\begin{array}{ll}
\lambda_j(D^2 u)  = 0 , \qquad & \mbox{ in } \Omega, \\[5pt]
u=g , \qquad & \mbox{ on } \partial \Omega.
\end{array}
\right.
\end{equation*}
\end{theorem}

We regard condition ($F$) as a geometric 
way to state \eqref{cond.Geo} without assuming that the boundary is smooth.
In section \ref{sect-HFG}, we discuss the relation within the different conditions on the boundary
in detail, we have that
$$
{\rm (H)}   \Rightarrow  {\rm (F)}    \Rightarrow {\rm (G)}.
$$

With the dynamic programming principle (DPP)
in mind we can obtain the following asymptotic mean value characterization of 
viscosity solutions to \eqref{1.1}. For the precise meaning of satisfying an asymptotic expansion
in the viscosity sense we refer to \cite{MPR} and Section \ref{sect.asymp}.

\begin{theorem}\label{eq.mena.value.charact}
Let $u$ be a continuous
function in a domain $\Omega\subset\mathbb{R}^N$.  
The asymptotic expansion 
$$
u(x) = \min_{dim(S)=j} \max_{
v\in S, |v|=1} \left\{ \frac{1}{2} u^\epsilon (x + \epsilon v) + \frac{1}{2} u^\epsilon (x - \epsilon v)
\right\} + o(\epsilon^2), \ \mbox{as }\epsilon \to 0,
$$
holds in the viscosity sense
if and only if
$$
\lambda_j ( D^2 u )  = 0
$$
in the viscosity sense. 
\end{theorem}

Our results can be easily extended to cover equations of the form
\begin{equation} \label{eq.sumas}
\sum_{i=1}^k \alpha_i \lambda_{j_i} =0
\end{equation}
with $\alpha_1 + ...+ \alpha_k =1$, $\alpha_i >0$ and
$\lambda_{j_1}\leq ... \leq \lambda_{j_k}$ any choice of $k$ 
eigenvalues of $D^2 u$ (not necessarily consecutive ones). 
In fact, once we fixed indexes $j_1,...,j_k$, we can just choose at random 
(with probabilities $\alpha_1,...,\alpha_k$) which game we play at each step 
(between the previously described games
that give $\lambda_{j_i}$ in the limit). In this case the DPP
reads as
$$
u^\epsilon (x) = \sum_{i=1}^k \alpha_i \left(
\inf_{
{dim}(S)=j_i
} \sup_{
v\in S, |v|=1}
\left\{ \frac{1}{2} u^\epsilon (x + \epsilon v) + \frac{1}{2} u^\epsilon (x - \epsilon v)
\right\} \right).
$$
Passing to the limit as $\epsilon \to 0$ we obtain a solution to \eqref{eq.sumas}.

In particular, we can handle equations of the form
\begin{equation*}
P^+_k (D^2 u) := \sum_{i=N-k+1}^N
\lambda_i(D^2 u)  = 0 ,
\quad
\mbox{and} \quad
P^-_k (D^2 u) := \sum_{i=1}^{k}
\lambda_i(D^2 u)  = 0,
\end{equation*}
or a convex combination of the previous two
\begin{equation*}
P^\pm_{k,l,\alpha} (D^2 u) := \alpha \sum_{i=N-k+1}^N
\lambda_i(D^2 u) + (1-\alpha) \sum_{i=1}^l
\lambda_i(D^2 u)   = 0.
\end{equation*}

These operators appear 
in \cite{Birin,Birin2,HL1,HL2} and in \cite{Sha,Wu} with connections
with geometry. See also \cite{caffa} for uniformly elliptic equations that
involve eigenvalues of the Hessian.

\begin{remark}{\rm
We can interchange the roles of Player I and Player II. In fact, 
consider a version of the game where 
the player who chooses the subspace $S$ of dimension $j$
is the one seeking to maximize the expected payoff
while the one who chooses the unitary vector 
wants to minimize the expected payoff.
In this case the game values will converge to a solution of the equation
$$
\lambda_{N-j+1} ( D^2 u )  = 0.
$$
Notice that the geometric condition on $\Omega$, ($F_j$) and ($F_{N-j+1}$), is 
also well suited to deal with this case.
}
\end{remark}

\medskip

The paper is organized as follows:
in Section \ref{sect-prelim} we collect
some preliminary results and include the definition of viscosity solutions;
in Section \ref{sect-concave}
we obtain the geometric interpretation of solutions to \eqref{1.1} stated
in Theorem \ref{teo.convex.envelope.k.intro};
in Section \ref{sect-conditions}
we prove Theorem~\ref{teo.main};
in Section \ref{sect-games} we prove our main results concerning the game,
Theorem \ref{teo.dpp} and Theorem \ref{teo.converge}; in Section \ref{sect-HFG}
we discuss the relation between the different geometric conditions on $\Omega$ and, finally,
in Section \ref{sect.asymp} we prove the asymptotic mean value characterization for 
solutions to \eqref{1.1}, Theorem \ref{eq.mena.value.charact}.

\section{Preliminaries} \label{sect-prelim}

We begin by stating the usual definition of a viscosity solution to \eqref{1.1}.
Here and in what follows $\Omega$ is a bounded domain in $\R^N$. We refer to
\cite{CIL} for general results on viscosity solutions.

First, let us recall the definition of the
lower semicontinuous envelope, $u_*$, and the upper semicontinuous envelope, $u^*$, of $u$, that is,
\[
u_*(x)=\sup_{r>0}\inf_{y\in B_r(x)} u(y)
\quad \text{and} \quad
u^*(x)=\inf_{r>0}\sup_{y\in B_r(x)} u(y).
\]

\begin{definition} \label{def.sol.viscosa.1}
A function  $u:\Omega \mapsto \mathbb{R}$  verifies
$$
\lambda_j ( D^2 u )  = 0
$$
\emph{in the viscosity sense} if
\begin{enumerate}
\item for every $\phi\in C^{2}$ such that $u_*-\phi $ has a strict
minimum at the point $x \in \Omega$
with $u_*(x)=\phi(x)$,
we have
$$
\lambda_j ( D^2 \phi (x) )  \leq 0.
$$

\item for every $ \psi \in C^{2}$ such that $ u^*-\psi $ has a
strict maximum at the point $ x \in {\Omega}$
with $u^*(x)=\psi(x)$,
we have
$$
\lambda_j( D^2 \psi (x) )  \geq 0.
$$
\end{enumerate}
\end{definition}
 
Now, we refer to \cite{HL1} for the following existence and uniqueness result for viscosity
solutions to \eqref{1.1}.

\begin{theorem}[\cite{HL1}] Let $\Omega$ be a smooth bounded domain in $\mathbb{R}^N$.
Assume that condition \eqref{cond.Geo} holds at every point on $\partial \Omega$. Then, for every
$g\in C(\partial \Omega)$, 
the problem
$$
\left\{
\begin{array}{ll}
\lambda_j(D^2 u)  = 0 , \qquad & \mbox{ in } \Omega, \\[5pt]
u=g , \qquad & \mbox{ on } \partial \Omega,
\end{array}
\right.
$$
has a unique viscosity solution $u\in C(\overline{\Omega})$.
\end{theorem}

We remark that for the equation $\lambda_j(D^2 u)  = 0$ there is a comparison principle.
A viscosity supersolution $\overline{u}$ (a lower semicontinuous function that verifies {\it (1)} 
in Definition \ref{def.sol.viscosa.1}) and  viscosity subsolution $\underline{u}$
(an upper semicontinuous function that verifies {\it (2)} 
in Definition \ref{def.sol.viscosa.1})
that are ordered as $\underline{u} \leq \overline{u}$ on $\partial \Omega$ are also ordered
as $\underline{u} \leq \overline{u}$ 
inside $\Omega$. This comparison principle holds without  assuming condition (H).

Condition \eqref{cond.Geo} allows us to construct a barrier at every point of the boundary.
This implies the continuity up to the boundary as stated above.
For the reader's convenience, let us include some details on the constructions of such barriers.
This calculations may help the reader to understand the interplay between the different conditions 
on the boundary of $\Omega$ that will be discussed in Section \ref{sect-HFG}.

For a given point 
 on the boundary (that we assume to be $x=0$)
we take coordinates according to $x_N$ in the direction of the normal vector 
and $(x_1,...,x_{N-1})$ in the tangent plane
in such a way that the main curvatures of the boundary $\kappa_1 \leq ... \leq \kappa_{N-1}$ corresponds to the directions
$(x_1,...,x_{N-1})$. That is, locally the boundary of $\Omega$ can be described as
$$
x_N = f(x_1,...,x_{N-1})
$$
with
$$
f(0,...,0) =0, \qquad \nabla f (0,...,0) =0.
$$
That is, locally we have that the boundary of $\Omega$ is given by 
$$
x_N- \frac12 \sum_{i=1}^{N-1} \kappa_i x_i^2 = o \left(\sum_{i=1}^{N-1}  x_i^2\right),
$$
and
$$
\begin{array}{l}
\displaystyle 
\Omega \cap B_r(0)  
= \Big\{ (x_1,...,x_N)\in B_r(0) \, : \, x_N - f(x_1,...,x_{N-1})  > 0 \Big\} \\[10pt]
\displaystyle \qquad 
= \left\{ (x_1,...,x_N)\in B_r(0) \, : \, x_N- \frac12 \sum_{i=1}^{N-1} \kappa_i x_i^2 > o \left(\sum_{i=1}^{N-1}  x_i^2\right) 
\right\}.
\end{array}
$$
for some $r>0$.

Now we take as candidate for a barrier a function of the form
$$
\overline{u}(x_1,...,x_N) = x_N- \frac12 \sum_{i=1}^{N-1} a_i x_i^2 - \frac12 bx_N^2,
$$
with  
$$
a_i = \kappa_i - \eta  \qquad \mbox{ and } \qquad b = \kappa_{N-j+1} -\eta.
$$
We have that
$$
D^2 (\overline{u}) =  \left( \begin{array}{ccccc}
-a_1 &  \dots & 0 & 0 \\
\vdots & \ddots &  & \vdots \\
0  &  & - a_{N-1} & 0 \\
 0 & \dots & 0 & - b 
\end{array} \right),
$$
and then the eigenvalues of $D^2 (\overline{u})$ are given by
\[
\begin{split}
\lambda_1 = -\kappa_{N-1}+\eta \leq \cdots \leq \lambda_{j-1} &= -\kappa_{N-j+1}+\eta= \\
 \lambda_j &= -\kappa_{N-j+1}+\eta \leq \cdots \leq \lambda_N = -\kappa_1+\eta. 
\end{split}
\]
We asked that condition \eqref{cond.Geo} holds, that implies, in particular, that
$$
\kappa_{N-j+1} >0,
$$
and therefore,
$$
\lambda_j (D^2 \overline{u}) = -\kappa_{N-j+1}+\eta <0
$$
for $\eta>0$ small enough. 

We also have 
$$
\overline{u}(x_1,...,x_N) >0 \qquad \mbox{ for } (x_1,...,x_N) \in \Omega \cap B(0,r)
$$
for $r$ small enough. To see this fact we argue as follows:
$$
\begin{array}{l}
\displaystyle
\overline{u}(x_1,...,x_N) = x_N- \frac12 \sum_{i=1}^{N-1} a_i x_i^2 - \frac12 bx_N^2 \\[10pt]
\displaystyle \qquad = x_N- f(x_1,...,x_{N-1}) + f(x_1,...,x_{N-1}) - \frac12 \sum_{i=2}^{n} a_i x_i^2 - \frac12 bx_1^2\\[10pt]
\displaystyle \qquad \geq  f(x_1,...,x_{N-1}) - \frac12 \sum_{i=1}^{N-1} \kappa_i x_i^2 + \frac{\eta}{2}\sum_{i=1}^{N} x_i^2  - 
\frac12 \kappa_{N-j+1}  x_N^2\\[10pt]
\displaystyle \qquad \geq   \frac{\eta}{2} \sum_{i=1}^{N} x_i^2  - 
\frac12 \kappa_{N-j+1}  x_N^2 + o\left(\sum_{i=1}^{N-1} x_i^2\right) .
\end{array}
$$
Since we assumed that $\kappa_{N-j+1} >0$ we have
\begin{align*}
\overline{u}(x_1,...,x_N)  
&\geq  \frac{\eta}{2} \sum_{i=1}^{N} x_i^2  - 
\frac12 \kappa_{N-j+1}  x_N^2 + o(\sum_{i=1}^{N-1} x_i^2)
\\
&\geq  \frac{\eta}{2} \sum_{i=1}^{N} x_i^2  - 
\frac12 \kappa_{N-j+1}  \left(\frac12 \sum_{i=1}^{N-1} \kappa_i x_i^2\right)^2 + o\left(\sum_{i=1}^{N-1} x_i^2\right)
\\
&\geq  \frac{\eta}{2} \sum_{i=1}^{N} x_i^2  - 
C \left(\sum_{i=1}^{N-1}  x_i^2\right)^2 + o\left(\sum_{i=1}^{N-1} x_i^2\right)>0
\end{align*}
for $(x_1,...,x_N) \in \Omega \cap B(0,r)$
with $r$ small enough. We also have that $\overline{u}(0)=0$
and at a point on $\partial\Omega\setminus\{0\}$
\[
\begin{split}
\overline{u}(x_1,...,x_N) &= x_N- \frac12 \sum_{i=1}^{N-1} a_i x_i^2 - \frac12 bx_N^2\\
&=\frac{\eta}{2} \sum_{i=1}^{N} x_i^2 + o\left(\sum_{i=1}^{N-1} x_i^2\right) >0.
\end{split}
\]

When looking for a subsolution we can do an analogous construction.
In this case we will use the condition $\kappa_j>0$.

\section{The geometry of convex/concave envelopes and the equation $\lambda_j =0$}
\label{sect-concave}

Let us describe a geometric interpretation of being a solution (the largest) to the equation
$$
\lambda_j( D^2 u )  = 0, \qquad \mbox{in } \Omega 
$$
with 
$u \leq g$ on $\partial \Omega$.

We begin with two special cases of Theorem \ref{teo.convex.envelope.k.intro}.

\subsection{$j=1$ and the convex envelope.} Let us start with the case $j=1$. We let
$H_1$ be the set of functions $v$ such that
$$v \leq g \qquad \mbox{on } \partial \Omega,$$
and have the following property: for every segment $D=(x_1,x_2)\subset \Omega$ it holds that 
$$
v \leq z \qquad \mbox{ in } D
$$
where $z$ is the linear function in $D$ with boundary values
$v|_{\partial D}$. In this case, the graph of $z$ is just the segment that joins
$(x_1, v(x_1))$ with $(x_2, v(x_2))$ and then we get
$$
v (t x_1 + (1-t) x_2) \leq t v(x_1) + (1-t) v(x_2) \qquad t \in (0,1).
$$
That is, $H_1$ is the set of convex functions in $\Omega$ that are less or
equal that $g$ on $\partial \Omega$.

Now we have

\begin{theorem} \label{teo.convex.envelope} Let
$$
u(x) = \sup_{v \in H_1} v(x).
$$
It turns out that $u$ is the largest viscosity solution to 
$$
\lambda_1( D^2 u )  = 0 \qquad \mbox{in } \Omega,
$$
with $u \leq g$ on $\partial \Omega$.
\end{theorem}

Notice that $u$ is just the convex envelope of $g$ in $\Omega$
and that this function is known to be twice differentiable almost everywhere
inside $\Omega$, \cite{Alex}.

\subsection{$j=N$ and the concave envelope.} Similarly, when one deals with $j=N$, we consider
$$
\lambda_N ( D^2 u )  = 0 \qquad \mbox{in } \Omega,
$$
with $u = g$ on $\partial \Omega$. We get that $v=-u$ is a solution to 
$$
\lambda_1 ( D^2 v )  = 0 \qquad \mbox{in } \Omega,
$$
with $v = - g$ on $\partial \Omega$. Hence $v=-u$ is the convex envelope of $-g$, that is, $u$ is
the concave envelope of $g$.

\subsection{$1<j<N$ and the convcave/convex envelope in affine spaces.} Let us consider $H_j$ the set of functions $v$ such that 
$$v \leq g \qquad \mbox{on } \partial \Omega,$$
and have the following property: for every $S$ affine 
of dimension $j$ and every $j-$dimensional domain 
$D \subset S \cap \Omega$ it holds that 
$$
v \leq z \qquad \mbox{ in } D
$$
where $z$ is the concave envelope of $v|_{\partial D}$ in $D$. 
Notice that, from our previous case, $j=N$, we have that the equation
for the convex envelope of $g$ in a $j-$dimensional domain $D$
is just $\lambda_j =0$.

Before we proceed with the proof of Theorem \ref{teo.convex.envelope.k.intro} 
we need to show the next lemma.
Notice that for a function $v \in H_j$ it could happen that $v^*$ does not
satisfy $v^* \leq g$ on $\partial \Omega$, nevertheless the main
condition in the definition of the set $ H_j$ still holds for  $v^*$.

\begin{lemma}
\label{lemma-star}
If $v \in H_j$ then for every $S$ affine 
of dimension $j$ and every $j-$dimensional domain 
$D \subset S \cap \Omega$ it holds that 
$$
v^* \leq z \qquad \mbox{ in } D
$$
where $z$ is the concave envelope of $v^*|_{\partial D}$ in $D$.
\end{lemma}

\begin{proof}
Suppose not.
Then, there exist $x\in\Omega$, an affine space $S$
of dimension $j$ and a $j-$dimensional domain 
$D \subset S \cap \Omega$ such that $x\in D$ and
$v^*(x)>z(x)$,
where $z:\ol D\to \R$ is the concave envelope of $v^*|_{\partial D}$ in $\ol D$. 
We consider $w=z+\eps$ for $\eps>0$ such that $v^*(x)>w(x)$.
We have that $w(y)>v^*(y)$ for every $y\in\partial D$. 
We suppose, without lost of generality, that $x=0$.

We know that there exists $x_k\in\Omega$ such that $x_k\to 0$ and $v(x_k)\to v^*(0)$.
We let $S_k=x_k+S$ and $D_k=(D+x_k)\cap \Omega$.
Now, we consider $r>0$ such that $B_{r}(0)\cap S\subset D$  and $B_{2r}(0)\subset \Omega$, if $|x_k|<r$
then $B_r(x_k)\subset D_k$.
Hence, $D_k$ is not empty for $k$ large enough, since we have that $x_k\in D_k$.

We consider $w_k:D_k\to \R$ given by $w_k(x)=w(x-x_k)$.
Since $v^*(0)>w(0)=w_k(x_k)$ and $v(x_k)\to v^*(0)$ we know that $v(x_k)>w_k(x_k)$ for $k$ large enough.
Since $w_k$ is concave, $v\in H_j$ and $v(x_k)>w_k(x_k)$ there exists $y_k\in\partial D_k$ such that $v(y_k)> w_k(y_k)$.
As $\partial D_k\subset \partial(D+x_k) \cup \partial\Omega$, by considering a subsequence we can assume that there exists $y$ such that $y_k\to y$, and $y_k\in\partial(D+x_k)$ for every $k$ or $y_k\in\partial\Omega$ for every $k$.

When $y_k\in\partial(D+x_k)$, we have that $y_k-x_k\in\partial D$ and hence $y\in \partial D$.
Since $v(y_k)> w_k(y_k)=w(y_k-x_k)$ and $w$ is continuous we obtain that 
\[
v^*(y)\geq \limsup_k v(y_k) \geq \limsup_k w(y_k-x_k)\geq w(y),
\]
which is a contradiction.

Now we consider the case when $y_k\in\partial\Omega$.
Since $y_k\in\ol D_k$, we have that $y\in\ol D$. 
If $y\in\partial D$ we can arrive to a contradiction as before.
If $y\in D$ then $y\in\Omega$ which is a contradiction since $y_k\in\partial\Omega$ and $y_k\to y$.
\end{proof}

Now, we are ready to prove the main theorem of this section.

\begin{proof} [Proof of Theorem~\ref{teo.convex.envelope.k.intro}]
First, let us show that every $v \in H_j$ is a viscosity 
subsolution to our problem.
In fact, we start mentioning that
$v \leq g$ on $\partial \Omega$. Concerning the equation, 
let $\phi \in C^2$ such that $\phi - v^*$ has a strict minimum at $x_0\in \Omega$ with $v^*(x_0)=\phi (x_0)$
($\phi$ touches $v^*$ from above at $x_0$) and assume, arguing by contradiction, that
$$
\lambda_j( D^2 \phi (x_0) )  < 0. 
$$
Therefore, there are $j$ orthogonal directions $v_1,...,v_j$ such that 
$$
\langle D^2\phi(x)v_i,v_i\rangle < 0.
$$
Notice that $\lambda_1( D^2 \phi (x_0)) \leq ... \leq \lambda_j( D^2 \phi (x_0)) <0$, therefore
the matrix $D^2 \phi (x_0))$ has at least $j$ negative eigenvalues.
Let us call $S $ the affine variety generated by $v_1,...,v_j$ that passes trough $x_0$.

Then we have, for any vector $w \in B_\delta (x_0) \cap S$ not null ($\delta$ small)
$$
v^* (x_0+w)  \leq \phi (x_0+ w) <  \phi (x_0) + \langle \nabla \phi (x_0) , w -x_0 \rangle.
$$
Therefore, we obtain that
$$
w \mapsto \phi (x_0) +  \langle \nabla \phi (x_0) , w-x_0 \rangle - \eps
$$
describes a function $z$ over the ball $B_\delta (x_0) \cap S$ with 
$v^*|_{\partial B_\delta (x_0) \cap S} \leq z|_{\partial B_\delta (x_0) \cap S}$ (for $\eps$ small),
such that
$$
z(x_0)= \phi (x_0)  - \eps<\phi (x_0)=v(x_0).
$$
A contradiction with the result in Lemma~\ref{lemma-star} since $v \in H_j$ and $z$ is linear and hence concave. 

This shows that every $v \in H_j$ is a subsolution and hence
$$
u(x) = \sup_{v \in H_j} v(x)
$$
is also a subsolution. 

Now, to show that $u$ is a supersolution we let $\phi \in C^2$ such that $\phi - u_*$ has a strict 
maximum at $x_0\in \Omega$ with $u_*(x_0)=\phi (x_0)$
($\phi$ touches $u_*$ from below at $x_0$) and assume, arguing by contradiction, that
$$
\lambda_j ( D^2 \phi (x_0) )  > 0. 
$$
Therefore, all the eigenvalues $\lambda_j ( D^2 \phi (x_0) ) \leq ... \leq
\lambda_N ( D^2 \phi (x_0) ) $ of $D^2 \phi (x_0)$ are strictly positive. 
Hence $\phi \in H_j$ in a small neighborhood of $x_0$ (every affine $S$
of dimension $j$
contains a direction $v$ such that $\langle D^2 \phi (x_0) v, v \rangle >0$).

Now, we take (for $\eps$ small) 
$$
\hat{u} (x) = \max \{u (x), \phi (x) + \eps \}
$$
and we obtain a function $\hat{u} \in H_j$ that verifies
$$
\hat{u} (z) = \max \{u (z), \phi (z) + \eps \}
 >  u(z) = \sup_{v \in H_j} v(z)
$$
for some $z$ close to $x_0$,
a contradiction.
\end{proof}

Hence, for a general $j$ we can say that the largest solution to our problem 
$$
\lambda_j( D^2 u )  = 0, \qquad \mbox{in } \Omega 
$$
with $u \leq g$ on $\partial \Omega$, is
the {\bf $j-$dimensional affine convex envelope} of $g$ inside $\Omega$.

\begin{remark} \label{rem.conc/convex} {\rm Notice that we can look at
the equation 
$$
\lambda_j =0
$$
from a dual perspective. 

Now, we consider $V_{N-j+1}$ the set of functions $w$ that are greater or equal than $g$ on $\partial \Omega$ and verify
the following property, for every
$T$ affine of dimension $N-j+1$ and any domain $D \subset T$, $w$ to be bigger or equal than
$z$ for every
$z $ a convex function in $D$ that is less or equal than $w$
on $\partial D$.

Let
$$
u(x) = \inf_{w \in V_{N-j+1}} w(x).
$$
Arguing as before, it turns out that $u$ is the smallest viscosity solution to 
$$
\lambda_j( D^2 u ) = 0, \qquad \mbox{in } \Omega 
$$
with $u \geq g$ on $\partial \Omega$.
}
\end{remark}

\section{Existence of continuous solutions} \label{sect-conditions}

In the previous section we showed 
existence and uniqueness of the largest/smallest viscosity solution to the PDE problem
$$
\lambda_j( D^2 u )  = 0, \qquad \mbox{in } \Omega 
$$
with 
$$u \leq g \ / \ u \geq g, \qquad \mbox{on }\partial \Omega.$$

Our main goal in this section is to show that under condition (G) 
on $\partial \Omega$ these functions coincide and then we have a 
solution $u$ that is continuous up to the boundary. Uniqueness and 
continuity inside $\Omega$
follow from the comparison principle for the equation $\lambda_j( D^2 u )  = 0$
proved in \cite{HL1}. In fact, for a solution that
is continuous on $\partial \Omega$, we have that $u^*$ is a subsolution and 
$u_*$ is a supersolution that verify $u^*=u_* =g$ on $\partial \Omega$
and then the comparison principle gives $u^* \leq u_*$ in $\Omega$. This fact
proves that $u=u^*=u_*$ is continuous.

Let us start by pointing out that when $\Omega$ does not satisfy condition (G)
then we have that ($G_j$) or ($G_{N-j+1}$) does not hold.

If $\Omega$ does not satisfy ($G_j$) then there exist $y\in\partial\Omega$, $r>0$, a sequences of points $x_n\in\Omega$ such that $x_n\to y$ and $S_n$ a sequence of affine subspaces of dimension $j$ such that $x_n\in S_n$ and
\[
S_n\cap\partial \Omega\cap B_r(y)=\emptyset.
\]

\begin{example} \label{ex-halfball} {\rm
The half-ball, that is, the domain $$\Omega=B_{1}(0)\cap \{x_2>0\}$$ in $\mathbb{R}^3$ does not
satisfy (G).
In fact, if we take $y=0\in \partial \Omega$, $r=\frac{1}{2}$, $x_n=(0,\frac{1}{n},0)$ and $S_n=x_n+\langle (1,0,0),(0,0,1)\rangle$ we have
\[
S_n\cap \Omega\cap B_r(y)=\emptyset
\]
for every $n$.

Now, let us show that \eqref{1.1} with $j=2$ 
does not have a continuous solution for a certain continuous boundary datum $g$.
We consider $g$ such that $g(x) \equiv 0$ for $x \in \partial B_{1}(0)\cap \{x_2>0\}$ and $g(0)=1$. 
Then, from our geometric characterization of solutions to the equation
$\lambda_2=0$ we obtain that there is no continuous solution to the Dirichlet problem
in $\Omega$ with datum $u=g$ on $\partial \Omega$. In fact, if 
such solution exists, then it must hold that
$$
u(0,a,0) \leq 0
$$
for every $a>0$. To see this, just observe that $u$ has to be less or equal than
$z\equiv 0$ that is the 
concave envelope of $g$ on the boundary of $\Omega \cap \{ x_2=a\}$.
Now, as $u$ is continuous we must have
$$
0\geq \lim_{a\searrow 0} u(0,a,0) = u(0,0,0) = g(0) =1
$$
a contradiction.}
\end{example}

With this example in mind we are ready to prove our main theorem.

\begin{proof}[Proof of Theorem~\ref{teo.main}]
Our goal is to show that \eqref{1.1} has a continuous solution for every boundary data $g$ if and only if $\Omega$ satisfy ($G$).

Let us start by proving that the condition is necessary.
We assume that $\Omega$ does not satisfies condition (G), hence ($G_j$) or ($G_{N-j+1}$)
does not hold.

If $\Omega$ does not satisfy ($G_j$) then there exist $y\in\partial\Omega$, $r>0$, a sequences of points $x_n\in\Omega$ such that $x_n\to y$ and $S_n$ a sequence of affine subspaces of dimension $j$ such that $x_n\in S_n$ and
\[
S_n\cap\partial \Omega\cap B_r(y)=\emptyset.
\]
We consider a continuous $g$ such that $g(y)=1$ and $g\equiv 0$ in $\partial\Omega\setminus B_r(y)$.
We assume there exists a solution $u$.
We have that $g\equiv 0$ in $S_n\cap\partial \Omega$ and hence $z\equiv 0$ is concave in $S_n\cap \ol\Omega$, we conclude that $u(x_n)\leq 0$ for every $n\in\N$.
Since $u(y)=g(y)=1$ we obtain that $u$ is not continuous.

If $\Omega$ does not satisfy ($G_{N-j+1}$) then we consider a continuous $g$ such that $g(y)=-1$ and $g\equiv 0$ in $\partial\Omega\setminus B_r(y)$.
As before we arrive to a contradiction by considering the characterization given in Remark~\ref{rem.conc/convex}.

We have proved that condition (G) is necessary. Now, let us show that if condition (G) holds we have a continuous solution for every continuous boundary datum $g$. To this end, we consider the largest viscosity solution to the our PDE,
$
\lambda_j( D^2 u )  = 0$ in $ \Omega $
with 
$u \leq g$ on $\partial \Omega$ that was constructed in the previous section.

We fix $y\in\partial\Omega$.
Given $\eps>0$, we want to prove that there exists $\delta>0$ such that $u(x)> g(y)-\eps$ for every $x\in\Omega\cap B_\delta(y)$. 
To prove this,
we will show there exists $\delta>0$ such that for every $x\in \Omega\cap B_\delta(y)$ and for every affine 
space $S$ of dimension $j$ through $x$, if we consider $D=\Omega\cap S$ and 
the concave envelope $z$ of $g|_{\partial D}$ in $D$, it holds that $$z(x)> g(y)-\eps.$$

Since $g$ is continuous, there exists $\ol\delta>0$ such that $|g(x)-g(y)|<\frac{\eps}{2}$ for every $x\in \partial\Omega\cap B_{\ol\delta}(y)$.
We consider $r\leq \ol\delta$ and $\delta>0$ such that condition $(G_j)$ is verified.
Given $x\in \Omega\cap B_\delta (y)$, for every affine space $S$ of dimension $j$ through $x$ there exists 
$v$ of norm one, a direction in $S$ such that
\begin{equation}
\label{condition}
\{x+tv\}_{t\in \R}\cap B_r(y)\cap \partial\Omega\neq\emptyset.
\end{equation}
We can consider the line segment $\ol{AB}$ contained in $\{x+tv\}_{t\in \R}$ such that $x\in\ol{AB}$, the interior of the segment is contained in $\Omega$ and $A,B\in\partial\Omega$.
Due to \eqref{condition} we can assume that $A\in B_r(y)\cap \partial\Omega$.

If $B\in B_{\ol \delta}(y)$, then, recalling that $A\in B_r(y)\subset B_{\ol\delta}(y)$, we have
\[
z(x)\geq\min\{g(A),g(B)\}>g(y)-\frac{\eps}{2}>g(y)-\eps.
\]
If $B\not\in B_{\ol \delta}(y)$, then $\dist(x,B)\geq \ol\delta-\delta$.
We have
\[
\begin{split}
z(x)
&\geq \frac{g(A) \dist(x,B)+g(B)\dist(x,A)}{\dist(x,A)+\dist(x,B)}\\
&\geq g(y)+\frac{(g(A)-g(y)) \dist(x,B)+(g(B)-g(y))\dist(x,A)}{\dist(x,A)+\dist(x,B)}\\
&\geq g(y)-\frac{|g(A)-g(y)| \dist(x,B)}{\dist(x,A)+\dist(x,B)}-\frac{|g(B)-g(y)|\dist(x,A)}{\dist(x,A)+\dist(x,B)}\\
&\geq g(y)-\frac{\eps}{2}-2\max|g|\frac{\dist(x,A)}{\dist(x,B)}.\\
\end{split}
\]
We know that $\dist(x,A)\leq r+\delta$.
If we take $\delta\leq r$, then, for $r$ small enough 
\[
z(x)\geq g(y)-\frac{\eps}{2}-2\max|g|\frac{2r}{\ol\delta-r}> g(y)-\eps
\]
as we wanted.

Analogously, taking into account that $\Omega$ verifies $(G_{N-j+1})$ and employing the characterization given in Remark~\ref{rem.conc/convex}, we can show that there exists $\delta>0$ such that $u(x)< g(y)+\eps$ for every $x\in\Omega\cap B_\delta(y)$.
In this way we obtain that $u$ is continuous on $\partial \Omega$ and hence in the whole $\overline{\Omega}$.
\end{proof}

\begin{example}{\rm
The domain $\Omega=B_{1.4}(0,0,1)\cup B_{1.4}(0,0,-1)$ in $\mathbb{R}^3$ 
that can be seen in Figure~\ref{fig:TwoSpheres} satisfy ($G_2$).
Hence, we have that the equation $\lambda_2=0$ has a solution in such domain.
Observe that the boundary is not smooth.
}
\end{example}

\begin{figure}
\includegraphics[scale=0.0625]{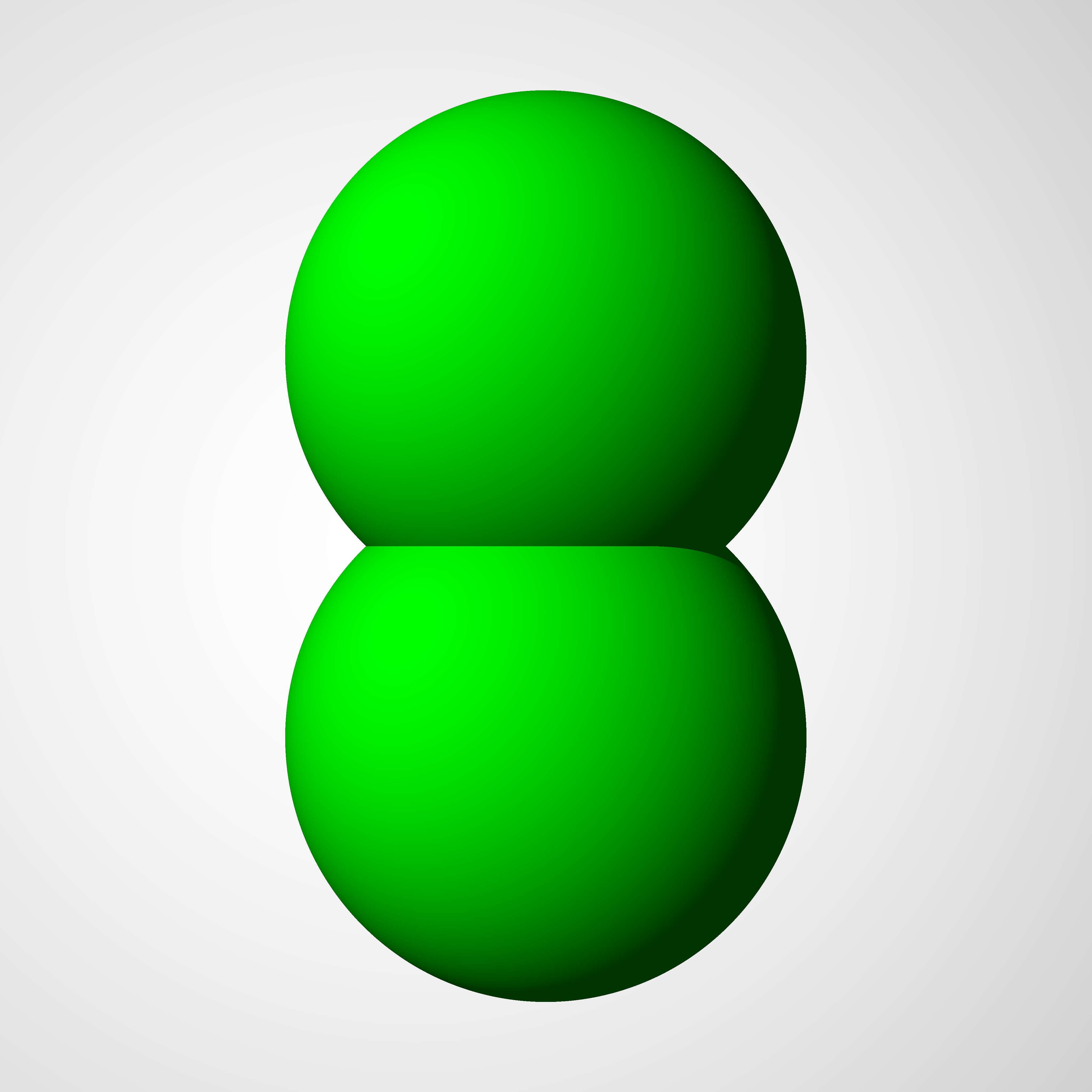}
\label{fig:TwoSpheres}
\caption{The domain $\Omega=B_{1.4}(0,0,1)\cup B_{1.4}(0,0,-1)$.}
\end{figure}

\section{Games} \label{sect-games}

In this section, we describe in detail the two-player zero-sum game that we call \textit{a random walk for
$\lambda_j$}.

Let $\Omega \subset\R^N$ be a bounded open set and fix $\eps>0$. 
A token is placed at $x_0\in\Omega $.
Player~I, the player seeking to minimize the final payoff,
chooses a subspace $S$ of dimension $j$
and then Player~II (who wants to maximize the expected payoff)
chooses one unitary vector, $v$, in the subspace $S$.
Then the position of the token
is moved to $x\pm \eps v$ with equal probabilities. 
After the first round, the game continues from $x_1$ according to the same rules.

This procedure yields a possibly infinite sequence of game states
$x_0,x_1,\ldots$ where every $x_k$ is a random variable.
The game ends when the token leaves $\Omega$, at this point the 
token will be in the boundary strip of
width $\epsilon$ given by
\[
\begin{split}\Gamma_\epsilon=
\{x\in {\mathbb{R}}^N \setminus \Omega \,:\,\dist(x,\partial \Omega )< \epsilon\}.
\end{split}
\]
We denote by $x_\tau \in \Gamma_\eps$ the first point in the
sequence of game states that lies in $\Gamma_\eps$, so that $\tau$
refers to the first time we hit $\Gamma_{\eps}$.
At this time the game
ends with the final payoff given by $g(x_\tau)$, where
$g:\Gamma_\eps
\to
\R$ is a given
continuous function that we call
\emph{payoff function}. Player~I earns $-g(x_\tau)$ while Player~II
earns $g(x_\tau)$.

A strategy $S_\I$ for Player~I is a function defined on the
partial histories that gives a $j-$dimensional subspace $S$ at every step of the game
\[
S_\I{\left(x_0,x_1,\ldots,x_k\right)}=S\in  Gr(j,\R^N).
\]
A strategy $S_\II$ for Player~II is a function defined on the
partial histories that gives a unitary vector in a prescribed $j-$dimensional subspace $S$ at every step of the game
\[
S_\II{\left(x_0,x_1,\ldots,x_k,S\right)}=v\in S.
\]

When the two players fix their strategies $S_I$ (the first player chooses 
a subspace $S$ at every step of the game) and $S_{II}$ (the second player chooses a unitary vector $v\in S$
at every step of the game) we can compute the expected outcome as follows:
Given the sequence $x_0,\ldots,x_k$ with $x_k\in\Om$ the next game position is distributed according to
the probability
\[
\begin{split}
\pi_{S_\I,S_\II}&(x_0,\ldots,x_k,{A})\\
&= \frac12 \delta_{x_k+\eps S_\II(x_0,\ldots,x_k,S_\I(x_0,\ldots,x_k))}(A)+
\frac12 \delta_{x_k-\eps S_\II(x_0,\ldots,x_k,S_\I(x_0,\ldots,x_k))}(A).
\end{split}
\]
By using the Kolmogorov's extension theorem and the one step transition probabilities, we can build a
probability measure $\mathbb{P}^{x_0}_{S_\I,S_\II}$ on the
game sequences. The expected payoff, when starting from $x_0$ and
using the strategies $S_\I,S_\II$, is
\begin{equation}
\label{eq:defi-expectation}
\mathbb{E}_{S_{\I},S_\II}^{x_0}\left[g(x_\tau)\right]=\int_{H^\infty} g(x_\tau) \ud
\mathbb{P}^{x_0}_{S_\I,S_\II}.
\end{equation}

The \emph{value of the game for Player I} is given by
\[
u^\eps_\I(x_0)=\inf_{S_\I}\sup_{S_{\II}}\,
\mathbb{E}_{S_{\I},S_\II}^{x_0}\left[g(x_\tau)\right]
\]
while the \emph{value of the game for Player II} is given by
\[
u^\eps_\II(x_0)=\sup_{S_{\II}}\inf_{S_\I}\,
\mathbb{E}_{S_{\I},S_\II}^{x_0}\left[g(x_\tau)\right].
\]
Intuitively, the values $u_\I(x_0)$ and $u_\II(x_0)$ are the best
expected outcomes
 each player can  guarantee when the game starts at
$x_0$. If $u^\eps_\I= u^\eps_\II$, we say that the game has a
value.

Let us observe that the game ends almost surely, then
the expectation \eqref{eq:defi-expectation} is well defined.
If we consider the square of the distance to a fix point in $\Gamma_\eps$, 
at every step, this values increases by 
at least
$\eps^2$ with probability $\frac12$.
As the distance to that point is bounded with a positive probability the game ends after a finite 
number of steps.
This implies that the game ends almost surely.

To see that the game has a value, we can consider $u^\eps$, a function that satisfies the DPP
\begin{equation*}
\left\{
\begin{array}{ll}
\displaystyle u^\eps (x) = \inf_{
{dim}(S)=j
} \sup_{
v\in S, |v|=1}
\left\{ \frac{1}{2} u^\eps (x + \eps v) + \frac{1}{2} u^\eps (x - \eps v)
\right\}  & x \in \Omega, \\[10pt]
u^\eps (x) = g(x)  & x \not\in \Omega.
\end{array}
\right.
\end{equation*}
The existence of such a function can be seen by Perron's method.
The operator given by the RHS of the DPP is in the hipoteses of the main result of \cite{QS}.

Now, we want to prove that $u^\eps=u^\eps_\I= u^\eps_\II$.
We know that $u^\eps_\I\geq u^\eps_\II$, to obtain the desired result, we will show that $u^\eps\geq u^\eps_\I$ and $u^\eps_\II \geq u^\eps$.

Given $\eta>0$ we can consider the strategy $S_\II^0$ for Player~II 
that at every step almost maximize 
$u^\eps (x_k + \eps v) + u^\eps (x_k - \eps v)$, that is
\[
S_\II^0{\left(x_0,x_1,\ldots,x_k,S\right)}=w\in S
\]
such that
\[
\begin{split}
\left\{ \frac{1}{2} u^\eps (x_k + \eps w) + \frac{1}{2} u^\eps (x_k - \eps w)\right\}\geq
\quad\quad\quad\quad\quad\quad\quad\quad\quad\quad\\
\sup_{
v\in S, |v|=1}
\left\{ \frac{1}{2} u^\eps (x_k + \eps v) + \frac{1}{2} u^\eps (x_k - \eps v) \right\}
-\eta 2^{-(k+1)} 
\end{split}
\]

We have
\[
\begin{split}
&\mathbb{E}_{S_\I, S^0_\II}^{x_0}[u^\eps(x_{k+1})-\eta 2^{-(k+1)}|\,x_0,\ldots,x_k]
\\
&\qquad \geq
\inf_{S , {dim}(S)=j}
 \sup_{v\in S, |v|=1}
\left\{ \frac{1}{2} u^\eps (x_k + \eps v) + \frac{1}{2} u^\eps (x_k - \eps v)
\right\}
\\
& \qquad\qquad -\eta 2^{-(k+1)}-\eta 2^{-(k+1)}
\\
&\qquad \geq u^\eps(x_k)-\eta 2^{-k},
\end{split}
\]
where we have estimated the strategy of Player I by $\inf$ and used the DPP.
Thus
\[
M_k=
u^\eps(x_k)-\eta2^{-k} 
\]
is a submartingale.
Now, we have
\begin{equation*}
\begin{split}
u^\eps_\II(x_0)
&=\sup_{S_\II}\inf_{S_{\I}}\,
\mathbb{E}_{S_{\I},S_\II}^{x_0}\left[g(x_\tau)\right]\\
&\geq\inf_{S_{\I}}\,
\mathbb{E}_{S_{\I},S^0_\II}^{x_0}\left[g(x_\tau)\right]\\
&\geq \inf_{S_\I} \liminf_{k\to\infty}\mathbb{E}_{S_{\I},S^0_\II}^{x_0}[M_{\tau\wedge k}]\\
&\geq \inf_{S_\I}\mathbb{E}_{S_{\I},S^0_\II}^{x_0}[M_0]=u^\eps(x_0)-\eta,
\end{split}
\end{equation*}
where $\tau\wedge k=\min(\tau,k)$, and we used the optional stopping theorem for $M_{k}$.
Since $\eta$ is arbitrary this proves that $u^\eps_\II \geq u^\eps$.
An analogous strategy can be consider for Player~I to prove that $u^\eps\geq u^\eps_\I$.

Now our aim is to pass to the limit in the values of the game
$$
u^\eps \to u, \qquad \mbox{as } \eps \to 0
$$
and obtain in this limit process a viscosity solution to \eqref{1.1}.

To obtain a convergent subsequence $u^\eps \to u$ we will use the following
Arzela-Ascoli type lemma. For its proof see Lemma~4.2 from \cite{MPRb}.

\begin{lemma}\label{lem.ascoli.arzela} Let $\{u^\eps : \overline{\Omega}
\to \R,\ \eps>0\}$ be a set of functions such that
\begin{enumerate}
\item there exists $C>0$ such that $\abs{u^\eps (x)}<C$ for
    every $\eps >0$ and every $x \in \overline{\Omega}$,
\item \label{cond:2} given $\eta>0$ there are constants
    $r_0$ and $\eps_0$ such that for every $\eps < \eps_0$
    and any $x, y \in \overline{\Omega}$ with $|x - y | < r_0 $
    it holds
$$
|u^\eps (x) - u^\eps (y)| < \eta.
$$
\end{enumerate}
Then, there exists  a uniformly continuous function $u:
\overline{\Omega} \to \R$ and a subsequence still denoted by
$\{u^\eps \}$ such that
\[
\begin{split}
u^{\eps}\to u \qquad\textrm{ uniformly in}\quad\overline{\Omega},
\end{split}
\]
as $\eps\to 0$.
\end{lemma}

So our task now is to show that the family $u^\eps$ satisfies the hypotheses of the previous lemma.

\begin{lemma}\label{lem.ascoli.arzela.acot} 
There exists $C>0$ independent of $\eps$ such that $$\abs{u^\eps (x)}<C$$ for
    every $\eps >0$ and every $x \in \overline{\Omega}$.
\end{lemma}

\begin{proof} We just observe that
$$
\min g \leq u^\eps (x) \leq \max g
$$
for every $x \in \overline{\Omega}$.
\end{proof}

To prove that $u^\eps$ satisfies second hypothesis we will have to make some geometric assumptions on the domain.
For our game with a given $j$ we will assume that $\Omega$ satisfies both ($F_j$) and ($F_{N-j+1}$).

Let us observe that for $j=1$ we assume ($F_N$), this condition can be read as follows.
Given $y\in\partial\Omega$ we assume that there exists
 $r>0$ such that for every $\delta>0$ there exists $v\in\R^N$ of norm 1 and $\theta>0$ such that
\begin{equation} \label{cond.geo1}
\{x\in\Omega\cap B_r(y): \langle v,x-y\rangle<\theta\}\subset B_\delta(y).
\end{equation}

\begin{lemma}\label{lem.ascoli.arzela.asymp} Given $\eta>0$ there are constants
    $r_0$ and $\eps_0$ such that for every $\eps < \eps_0$
    and any $x, y \in \overline{\Omega}$ with $|x - y | < r_0 $
    it holds
$$
|u^\eps (x) - u^\eps (y)| < \eta.
$$
\end{lemma}

\begin{proof} 
The case $x,y \in \Gamma_\eps$ follows from the uniformity continuity of $g$ in $\Gamma_\eps$.
For the case $x,y \in \Omega $ we argue as follows.
We fix the strategies $S_\I,
S_\II$ for the game starting at $x$. We define a virtual game
starting at $y$. We use the same random steps as
the game starting at $x$. Furthermore, the players adopt
their strategies $S_\I^v, S_\II^v$ from the game starting at $x$,
that is, when the game position is $y_k$ a player make the choices that would have taken at $x_k$ in the game starting at $x$.
We proceed in this way until for the first time $x_k \in
\Gamma_\eps$ or $y_k \in \Gamma_\eps$. At that point we have $|x_k
-y_k| =|x-y|$, and the desired estimate follow from the one for 
$x_k \in \Omega$, $y_k\in \Gamma_\eps$ or for $x_k, y_k \in
\Gamma_\eps$.

Thus, we can concentrate on the case $x\in \Omega$ and $y \in \Gamma_\eps$.
Even more, we can assume that $y\in\partial\Omega$.
If we have the bound for those points we can obtain a bound for a point $y \in \Gamma_\eps$ 
just by considering $z\in\partial\Omega$ in the line segment between $x$ and $y$.

In this case we have 
$$
u_\eps (y) = g(y),
$$
and we need to obtain a bound for $u_\eps (x)$.

First, we deal with $j=1$.
To this end we just observe that, for any possible strategy
of the players (that is, for any possible choice of the direction $v$ at every point)
we have that the projection of $x_n$ in the direction of the a fixed vector $w$ of norm 1, 
$$
\left\langle x_n-y, w\right\rangle 
$$
is a martingale.
We fix $r>0$ and consider $x_\tau$, the first time $x$ leaves $\Omega$ or $B_r(y)$.
Hence
$$
\mathbb{E} \left\langle x_\tau-y, w \right\rangle \leq \left\langle x-y, w \right\rangle \leq d(x,y)  < r_0.
$$
From the geometric assumption on $\Omega$, we have that $\left\langle x_n-y, w\right\rangle \geq -\eps$.
Therefore
$$
\mathbb{P} \left( \left\langle x_\tau-y, w \right\rangle  > r_0^{1/2} \right) r_0^{1/2} -
\left(1-\mathbb{P} \left(  \left\langle x_\tau-y, w \right\rangle  > r_0^{1/2} \right) \right)\eps < r_0.
$$
Then, we have (for every $\eps$ small enough)
$$
\mathbb{P} \left( \left\langle x_\tau-y, w \right\rangle  > r_0^{1/2} \right)  <  2 r_0^{1/2}.
$$
Then, \eqref{cond.geo1} implies that 
given $\delta>0$ we can conclude that
$$
\mathbb{P} ( d ( x_\tau, y)  > \delta )  <  2  r_0^{1/2}.
$$
by taking $r_0$ small enough and a appropriate $w$.

When $d(x_\tau, y)\leq \delta$, the point $x_\tau$ is actually the point where the process have leaved $\Omega$.
Hence,
$$
\begin{array}{l}
|u_\eps (x) - g(y)| \\[10pt]
\displaystyle \leq \mathbb{P} (d(x_\tau, y)\leq \delta )|g (x_\tau) - g(y)| + 
\mathbb{P} (d(x_\tau, y)> \delta ) 2 \max g \\[10pt]
\displaystyle
\leq \sup_{x_\tau\in B_\delta(y)}|g (x_\tau) - g(y)| +4r_0^{1/2} \max g <\eta
\end{array}
$$
if $r_0$ and $\delta$ are small enough.

For a general $j$ we can proceed in the same way.
We have to make some extra work to argue that the points $x_n$ that appear along the argument belong to $T_\lambda$.
If $r_0<\lambda$ we have that $x\in T_\lambda$, so if we make sure that at every move $v\in T$ we will have that the game sequence will be contained in $x+T\subset T_\lambda$. 

Recall that here we are assuming both ($F_j$) and ($F_{N-j+1}$) are satisfied.
We can separate the argument into two parts. We will prove on the one hand that $u_\eps (x) - g(y)<\eta$ 
and on the other that $g(y)-u_\eps (x)<\eta$.
For the first inequality we can make extra assumptions on the strategy for Player~I, and for the second one we can do the same with Player~II. 

Since $\Omega$ satisfies ($F_j$), Player~I can make sure that at every move $v$ belongs to $T$ by selecting $S=T$.
This proves the upper bound $u_\eps (x) - g(y)<\eta$.
On the other hand, since $\Omega$ satisfy ($F_{N-j+1}$), Player~II will be able to select $v$ in a space $S$ of dimension $j$ and hence he can always choose $v\in S\cap T$ since 
\[
\dim(T)+\dim(S)=N-j+1+j=N+1>N.
\]
This shows the lower bound $g(y)-u_\eps (x)<\eta$.
\end{proof}

From Lemma \ref{lem.ascoli.arzela.acot} and Lemma \ref{lem.ascoli.arzela.asymp}
we have that the hypotheses of the Arzela-Ascoli type lemma, Lemma \ref{lem.ascoli.arzela},
are satisfied. Hence we have obtained uniform convergence of $u^\eps$ along a subsequence.

\begin{corollary} \label{corol.converge}
Let $u^\eps$ be the values of the game. Then, along a subsequence,
\begin{equation}\label{eq.converge.22}
u^\eps \to u, \qquad \mbox{ as } \eps \to 0,
\end{equation}
uniformly in $\overline{\Omega}$. 
\end{corollary}

Now, let us prove that any possible limit of $u^\eps$ is a viscosity solution to
the limit PDE problem. 

\begin{theorem} \label{teo.converge.22}
Any uniform limit of the values of the game $u^\eps$, $u$, is a viscosity
solution to
\begin{equation}\label{1.1.teo.22}
\left\{
\begin{array}{ll}
\lambda_j(D^2 u)  = 0 , \qquad & \mbox{ in } \Omega, \\[5pt]
u=g , \qquad & \mbox{ on } \partial \Omega.
\end{array}
\right.
\end{equation}
\end{theorem}

\begin{proof}
First, we observe that since $u^\eps =g$ on $\partial \Omega$ we obtain,
form the uniform convergence, that $u =g$ on $\partial \Omega$.
Also, notice that Lemma \ref{lem.ascoli.arzela}  gives that a uniform limit of 
$u^\eps$ is a continuous function. Hence, we avoid
the use of $u^*$ and $u_*$ in what follows.

To check that $u$ is a viscosity solution to $\lambda_j(D^2 u)  = 0$ in $\Omega$, in the
sense of Definition \ref{def.sol.viscosa.1}, let $\phi\in C^{2}$ be such that $ u-\phi $ has a strict
minimum at the point $x \in \Omega$  with $u(x)=
\phi(x)$. We need to check that
$$
\lambda_j ( D^2 \phi (x) )  \leq 0.
$$
As $u^\eps \to u$ uniformly in $\overline{\Omega}$ we have the existence of a sequence
$x_\eps$ such that $x_\eps \to x$ as $\eps \to 0$ and 
$$
u^\eps (z) - \phi (z) \geq u^\eps (x_\eps) - \phi (x_\eps) - \eps^3
$$
(remark that $u^\epsilon$ is not continuous in general). 
As $u^\eps$ is a solution to
$$
u^\epsilon (x) = \inf_{
 {dim}(S)=j
} \sup_{
v\in S, |v|=1}
\left\{ \frac{1}{2} u^\eps (x + \eps v) + \frac{1}{2} u^\eps (x - \eps v)
\right\} 
$$
we obtain that $\phi$ verifies the inequality
$$
0 \geq \inf_{
{dim}(S)=j
} \sup_{
v\in S, |v|=1}
\left\{ \frac{1}{2} \phi (x_\eps + \eps v) + \frac{1}{2} \phi (x_\eps - \eps v)
- \phi (x_\eps)
\right\} - \eps^3.
$$

Now, consider the Taylor
expansion of the second order of $\phi$
\[
\phi(y)=\phi(x)+\nabla\phi(x)\cdot(y-x)
+\frac12\langle D^2\phi(x)(y-x),(y-x)\rangle+o(|y-x|^2)
\]
as $|y-x|\rightarrow 0$. Hence, we have
\begin{equation} \label{22.q}
\phi(x+\eps v)=\phi(x)+\eps \nabla\phi(x)\cdot v
+\eps^2 \frac12\langle D^2\phi(x)v,v\rangle+o(\eps^2)
\end{equation}
and
\begin{equation} \label{33.q}
\phi(x- \eps v)=\phi(x) - \eps \nabla\phi(x)\cdot v
+\eps^2 \frac12\langle D^2\phi(x)v,v\rangle+o(\epsilon^2).
\end{equation}
Hence, using these expansions we get
$$
\frac{1}{2} \phi (x_\eps + \eps v) + \frac{1}{2} \phi (x_\eps - \eps v)
- \phi (x_\eps) = \frac{\eps^2}2 \langle D^2 \phi (x_\eps)v, v \rangle + o(\eps^2),
$$
and then we conclude that
$$
0 \geq \eps^2 \inf_{
{dim}(S)=j
} \sup_{
v\in S, |v|=1}
\left\{ \frac12 \langle D^2 \phi (x_\eps)v, v \rangle 
\right\} + o(\eps^2).
$$
Dividing by $\eps^2$ and passing to the limit as $\eps \to 0$ we get
$$
0 \geq  \inf_{
 {dim}(S)=j
} \sup_{
v\in S, |v|=1}
\left\{  \langle D^2 \phi (x)v, v \rangle 
\right\} ,
$$
that is equivalent to
$$
0 \geq \lambda_j (D^2 \phi (x)) 
$$
as we wanted to show. 

The reverse inequality when a smooth function $\psi$
touches $u$ from below can be obtained in a similar way.
\end{proof}

\begin{remark}{\rm Since there is uniqueness of viscosity solutions
to the limit problem \eqref{1.1.teo.22} (uniqueness holds for every domain without
any geometric restriction once we have existence of a continuous solution) we obtain that the uniform limit
$$
\lim_{\eps \to 0} u^\eps = u
$$
exists (not only along a subsequence).
}
\end{remark}

\section{Geometric conditions on $\partial \Omega$} \label{sect-HFG}

Now, our goal is to analyze the relation between the different conditions on $\partial \Omega$.
We have introduced in this paper three different conditions: 

(H) that involve the curvatures of $\partial \Omega$ and hence requires smoothness, this condition
was used in \cite{HL1} to obtain existence of a continuous viscosity solution
to \eqref{1.1}.

(F) that is given by ($F_j$) and ($F_{N-j+1}$). This condition was used to obtain
convergence of the values of the game. 

(G) that was proved to be equivalent to the solvability
of \eqref{1.1} for every continuous datum $g$.

We will show that
$$
{\rm (H)}   \Rightarrow  {\rm (F)}    \Rightarrow {\rm (G)}.
$$

\subsection{(H) implies ($F_j$)} Let us show that the condition $\kappa_{N-j+1}>0$ in (H) implies ($F_j$).
We consider $T=\langle x_{N-j+1},\dots,x_N\rangle$ (note that this is a subspace of dimension $j$), $v=x_N$ and $r$ as above.
We want to show that for every $\delta>0$ there exists $\lambda>0$ and $\theta>0$ such that
\begin{equation} 
\{x\in\Omega\cap B_r(y)\cap T_\lambda: \langle v,x-y\rangle<\theta\}\subset B_\delta(y).
\end{equation}
We have to choose $\lambda$ and $\theta$ such that for $x$ with $\|x\|>\delta$, 
\[
\|(x_1,\dots,x_{N-j})\|<\lambda\]
and
\[
x_N- \frac12 \sum_{i=1}^{N-1} \kappa_i x_i^2 > o \left(\sum_{i=1}^{N-1}  x_i^2\right),
\]
 it holds that 
 $$x_N>\theta.$$
Let us prove this fact. We have
\[
\begin{split}
x_N 
&>\frac12 \sum_{i=1}^{N-1} \kappa_i x_i^2+ o \left(\sum_{i=1}^{N-1}  x_i^2\right)\\
&\geq \frac12 \sum_{i=1}^{N-j} \kappa_i x_i^2+ \frac12\sum_{i=N-j+1}^{N-1} \kappa_i x_i^2+ o \left(\sum_{i=1}^{N-1}  x_i^2\right)\\
&\geq -C_1 \sum_{i=1}^{N-j}x_i^2+ C_2 \sum_{i=1}^{N-1}  x_i^2+ o \left(\sum_{i=1}^{N-1}  x_i^2\right)\\
&\geq -C_1 \lambda^2+ C_2\delta^2 + o \left(\sum_{i=1}^{N-1}  x_i^2\right)>\theta\\
\end{split}
\]
for $r$, $\lambda$ and $\theta$ small enough (for a given $\delta$).

\subsection{(F) implies (G)}
We proved that (F) implies existence of a continuous viscosity solution
to \eqref{1.1} (that was obtained as the limit of the values of the game
described in Section \ref{sect-games}). Notice that we have proved that (G) is equivalent to the 
existence of a continuous solution to \eqref{1.1} for every continuous datum $g$. Then,
we deduce that (F) implies (G). 

\medskip

The same argument can be used to show that (H) implies (G) directly.

\subsection{(H) implies (G)} We use again that (G) is equivalent to the 
existence of a continuous solution to \eqref{1.1} for every continuous datum $g$
and that in \cite{HL1} it is proved that (H) implies existence of a continuous viscosity solution
to \eqref{1.1} thanks to the construction of the barriers described in Section \ref{sect-prelim}.
Hence we can deduce that (H) implies (G).

\section{Asymptotic mean value formulas} \label{sect.asymp}

A well known fact 
states that $u$ is harmonic, that is $u$ verifies $\Delta u =0$, 
if and only if it verifies the mean value property
$$
u(x) =
\frac{1}{|B_\varepsilon (x) |} \int_{B_\varepsilon (x) } u(y) \, dy.
$$
For a mean value property for the $p-$Laplacian we refer to \cite{MPR} and \cite{LM}.

Here our main concern is to obtain mean value properties for our equation 
\begin{equation}\label{eq.1}
\lambda_j(X)  = 0 ,
\end{equation}
where for a matrix $X$, $\lambda_1(X) \leq ... \leq \lambda_N (X)$
stand for the ordered eigenvalues of $X$. 

Now, as we used before, we note that this equation can be written as
$$
\lambda_j (X) = \min_{dim(S)=j} \max_{v\in S, |v|=1}
  \langle X v, v \rangle  = \lambda_j ,
$$
where the minimum is taken among all possible subspaces of $\R^N$ with dimension $j$
and for each $S$ the maximum is taken among unitary vectors $v$ in $S$.
In fact, one can easily check that for any symmetric matrix $X$ its holds that
$$
\langle X v, v \rangle = \sum_{i=1}^N (a_i)^2 \lambda_i 
$$
if $\lambda_1\leq ... \leq \lambda_N$ are the eigenvalues of $X$, with
corresponding orthonormal eigenvectors $v_1,...,v_N$ and 
$v = \sum_{i=1}^N a_i v_i$. From this expression it can be easily
deduced that the $j-$st eigenvalue verifies
$$
 \min_{dim(S)=j} \max_{v\in S, |v|=1}
  \langle X v, v \rangle  = \lambda_j.
$$

Recall that in Section \ref{sect-prelim} we have introduced the definition of
viscosity solutions to \eqref{1.1}, Definition \ref{def.sol.viscosa.1}.

Now, we introduce the definition of
our asymptotic expansions in
\lq\lq a viscosity sense\rq\rq .
 As is the case in the theory of viscosity solutions,
we test the expansions of a function $u$ against test functions
$\phi$ that touch $u$ from below or above at a particular point.
As above, here $S$ denotes a subspace of dimension $j$.

\begin{definition} \label{def.sol.viscosa.asymp}
A continuous function  $u$  verifies
$$
u(x) =  \min_{dim(S)=j} \max_{v\in S, |v|=1}
  \left\{ \frac12 u(x+\eps v) + \frac12 u(x-\eps v)  \right\}
+ o(\eps^2), \quad \mbox{as }\eps \to 0,
$$
in the \emph{viscosity sense} if
\begin{enumerate}
\item for every $\phi\in C^{2}$ such that $ u-\phi $ has a strict
minimum at the point $x \in \overline \Omega$  with $u(x)=
\phi(x)$, we have
$$
0 \geq - \phi (x) +
 \min_{dim(S)=j} \max_{v\in S, |v|=1}
  \left\{ \frac12 \phi(x+\eps v) + \frac12 \phi (x-\eps v)  \right\}+ o(\eps^2).
$$

\item for every $ \psi \in C^{2}$ such that $ u-\psi $ has a
strict maximum at the point $ x \in \overline{\Omega}$ with $u(x)=
\psi(x)$, we have
$$
0 \leq - \psi (x) +
 \min_{dim(S)=j} \max_{v\in S, |v|=1}
  \left\{ \frac12 \psi(x+\eps v) + \frac12 \psi (x-\eps v)  \right\}+ o(\eps^2).
$$
\end{enumerate}
\end{definition}

Theorem \ref{eq.mena.value.charact} says that Definitions \ref{def.sol.viscosa.1} and \ref{def.sol.viscosa.asymp} are equivalent.
Therefore we have an asymptotic mean value characterization of solutions to \eqref{eq.1}.

\begin{proof}[Proof of Theorem \ref{eq.mena.value.charact}]
First, assume that the
asymptotic expansion 
$$
u(x) =  \min_{dim(S)=j} \max_{v\in S, |v|=1}
  \left\{ \frac12 u(x+\eps v) + \frac12 u(x-\eps v)  \right\}
+ o(\eps^2), \quad \mbox{as }\eps \to 0,
$$
holds in the viscosity sense. We have to show that 
$$
\lambda_j( D^2 u )  = 0
$$
also in the viscosity sense.

To this end take a point $x\in \Omega$ and a $C^2$-function $\phi$
such that 
$ u-\phi $ has a strict
minimum at the point $x \in \overline \Omega$  with $u(x)=
\phi(x)$. 

Since we assumed that the asymptotic expansion 
holds, from Definition \ref{def.sol.viscosa.asymp}
we have
\begin{equation}\label{pp}
0 \geq - \phi (x) +
 \min_{dim(S)=j} \max_{v\in S, |v|=1}
  \left\{ \frac12 \phi(x+\eps v) + \frac12 \phi (x-\eps v)  \right\}+ o(\eps^2).
\end{equation}

Consider the Taylor
expansion of the second order of $\phi$
\[
\phi(y)=\phi(x)+\nabla\phi(x)\cdot(y-x)
+\frac12\langle D^2\phi(x)(y-x),(y-x)\rangle+o(|y-x|^2)
\]
as $|y-x|\rightarrow 0$. Hence, we have
\begin{equation} \label{22}
\phi(x+\eps v)=\phi(x)+\eps \nabla\phi(x)\cdot v
+\eps^2 \frac12\langle D^2\phi(x)v,v\rangle+o(\eps^2)
\end{equation}
and
\begin{equation} \label{33}
\phi(x- \eps v)=\phi(x) - \eps \nabla\phi(x)\cdot v
+\eps^2 \frac12\langle D^2\phi(x)v,v\rangle+o(\eps^2).
\end{equation}
Adding these two expansions and using \eqref{pp} we arrive to 
$$
0 \geq \min_{dim(S)=j} \max_{v\in S, |v|=1}\left\{ \eps^2 \frac12\langle D^2\phi(x)v,v\rangle 
\right\}
+ o(\eps^2).
$$
Dividing by $\eps^2$ and taking limit as $\eps \to 0$ we get
$$
0 \geq \min_{dim(S)=j} \max_{v\in S, |v|=1} \left\{ \langle D^2\phi(x)v,v\rangle 
\right\} = \lambda_j(D^2 \phi(x)) = P_j (D^2 \phi (x)),
$$
as we wanted to show.

 The argument for the  case of supersolutions is analogous (just consider $\psi$ as in Definition 
 \ref{def.sol.viscosa.asymp} and reverse the inequalities).
 
 Now, let us take a viscosity solution to \eqref{eq.1}, and $x\in \Omega$ and a $C^2$ test function $\phi$ 
such that 
$ u-\phi $ has a strict
minimum at the point $x \in \overline \Omega$  with $u(x)=
\phi(x)$. 

Using again the Taylor expansions \eqref{22} and \eqref{33} we obtain
$$
\begin{array}{l}
\displaystyle
- \phi (x) +
 \min_{dim(S)=j} \max_{v\in S, |v|=1}
  \left\{ \frac12 \phi(x+\eps v) + \frac12 \phi (x-\eps v)  \right\} 
    \\[10pt]
\displaystyle \qquad = \min_{dim(S)=j} \max_{v\in S, |v|=1} \left\{ \langle D^2\phi(x)v,v\rangle 
\right\}  + o (\eps^2).
\end{array} 
$$
Using that $u$ is a viscosity solution to \eqref{eq.1}, from Definition \ref{def.sol.viscosa.1} we get
$$
0 \geq \min_{dim(S)=j} \max_{v\in S, |v|=1} \left\{ \langle D^2\phi(x)v,v\rangle 
\right\} = P_j (D^2 \phi (x)),
$$
and hence we conclude that
$$
0 \geq - \phi (x) +
 \min_{dim(S)=j} \max_{v\in S, |v|=1}
  \left\{ \frac12 \phi(x+\eps v) + \frac12 \phi (x-\eps v)  \right\}+ o(\eps^2),
$$
as we wanted to show.

The argument with $\psi$ is analogous.
\end{proof}

\medskip

{\bf Acknowledgements.} Partially supported by CONICET grant PIP GI No 11220150100036CO
(Argentina), by  UBACyT grant 20020160100155BA (Argentina) and by MINECO MTM2015-70227-P
(Spain).

%%%%%%%%%%%%%%%%%%%%%%%%%%%%%%%%%%%%%%%%%%%%%%%%%%%%%%%%%%%%%%%%
%%%%%%%%%%%%%%%%%%%%%%%%%%%%%%%%%%%%%%%%%%%%%%%%%%%%%%%%%%%%%%%%
%%%%%%%%%%%%%%%%%%%%%%%%%%%%%%%%%%%%%%%%%%%%%%%%%%%%%%%%%%%%%%%%
%%%%%%%%%%%%%%%%%%%%%%%%%%%%%%%%%%%%%%%%%%%%%%%%%%%%%%%%%%%%%%%%

\end{document}